\documentclass{amsart}

\usepackage{amssymb}
\usepackage{amsthm,amsmath}
\usepackage{atbegshi}
\usepackage{tikz}
\usepackage{hyperref}

\begin{document}

\title{A note on the total number of cycles of even and odd permutations}

\author[jskim]{Jang Soo Kim}
\email{jskim@kaist.ac.kr}
\thanks{The author is supported by the grant ANR08-JCJC-0011.}

\begin{abstract}
  We prove bijectively that the total number of cycles of all even
  permutations of $[n]=\{1,2,\ldots,n\}$ and the total number of cycles of
  all odd permutations of $[n]$ differ by $(-1)^n(n-2)!$, which was stated
  as an open problem by Mikl\'{o}s B\'{o}na.  We also prove bijectively the
  following more general identity:
$$\sum_{i=1}^n c(n,i)\cdot i \cdot (-k)^{i-1} = (-1)^k k! (n-k-1)!,$$
where $c(n,i)$ denotes the number of permutations of $[n]$ with $i$ cycles. 
\end{abstract}

\maketitle

\newtheorem{thm}{Theorem}
\theoremstyle{definition}
\newtheorem{example}{Example}
\newcommand\sgn{\operatorname{sign}}
\newcommand\Fix{\operatorname{Fix}}

\section{Introduction}
\label{sec:introduction}

Let $c(n,i)$ denote the number of permutations of $[n]=\{1,2,\ldots,n\}$ with $i$ cycles. 
The following equation is well known; for example see \cite{Bona,Stanley1997}:
\begin{equation}
  \label{eq:2}
\sum_{i=1}^n c(n,i) x^i = x(x+1) \cdots (x+n-1).
\end{equation}

Let $n$ and $k$ be positive integers with $k<n$. 
By differentiating \eqref{eq:2} with respect to $x$ and substituting $x=-k$, we get the following:
\begin{equation}
  \label{eq:4}
\sum_{i=1}^n c(n,i)\cdot i \cdot (-k)^{i-1} = (-1)^k k! (n-k-1)!.
\end{equation}
In particular, if $k=1$, then \eqref{eq:4} implies the following theorem.

\begin{thm}\label{thm:1}
The total number of cycles of all even permutations of $[n]$ and the 
total number of cycles of all odd permutations of $[n]$ differ by $(-1)^n(n-2)!$.  
\end{thm}

The problem of finding a bijective proof of Theorem~\ref{thm:1} was proposed by Mikl\'{o}s B\'{o}na and it has been added to \cite{StanleyEC1se} as an exercise (private communication with Richard Stanley and Mikl\'{o}s B\'{o}na).
In this note, we prove Theorem~\ref{thm:1} bijectively by finding a sign-reversing involution. We also prove \eqref{eq:4} bijectively.

\section{Bijective proofs}
\label{sec:bijective-proofs}

Recall the lexicographic order on the pairs of integers, that is,  $(i_1,j_1) \leq  (i_2,j_2)$ if and only if $i_1 < i_2$, or $i_1=i_2$ and $j_1\leq j_2$. Note that this is a linear order.

Let $T(n)$ denote the set of pairs $(\pi,C)$ where $\pi$ is a permutation of $[n]$ and $C$ is a cycle of $\pi$. Then Theorem~\ref{thm:1} is equivalent to the following:
\begin{equation}
  \label{eq:1}
\sum_{(\pi,C)\in T(n)} \sgn(\pi) =(-1)^n(n-2)!.
\end{equation}

\begin{proof}[Proof of Theorem~\ref{thm:1}]
We define a map $\phi:T(n)\rightarrow T(n)$ as follows. Let $(\pi,C)\in T(n)$. 

\textbf{\textsc{Case 1:}} $C$ contains at most $n-2$ integers. Let $(i,j)$ be the smallest pair in lexicographic order for distinct integers $i$ and $j$ which are not contained in $C$. 
Then we define $\phi(\pi,C)=(\tau_{ij}\pi,C)$, where $\tau_{ij}$ is the transposition exchanging $i$ and $j$. 

\textbf{\textsc{Case 2:}} $C$ contains at least $n-1$ integers. If $C$ does not contain $1$, then we define $\phi(\pi,C)=(\pi,C)$. If $C$ contains $1$, then we have either $\pi=(a_0)(1,a_1,a_2,\ldots,a_{n-2})$ or $\pi=(1,a_0,a_1,\ldots,a_{n-2})$ in cycle notation for some integers $a_i$. Let $\pi'=(1,a_0,a_1,\ldots,a_{n-2})$ if $\pi=(a_0)(1,a_1,a_2,\ldots,a_{n-2})$, and $\pi'=(a_0)(1,a_1,a_2,\ldots,a_{n-2})$ if $\pi=(1,a_0,a_1,\ldots,a_{n-2})$.
We define $\phi(\pi,C)=(\pi',C')$, where $C'$ is the cycle of $\pi'$ containing $1$.

Let us define the \emph{sign} of $(\pi,C)\in T(n)$ to be $\sgn(\pi)$.  It
is easy to see that $\phi$ is a sign-reversing involution on $T(n)$ whose
fixed points are precisely those $(\pi,C)\in T(n)$ such that $1$ forms a
$1$-cycle and the rest of the integers form an $(n-1)$-cycle, which is $C$. Since
there are $(n-2)!$ such fixed points of $\phi$ which all have sign
$(-1)^n$, we get \eqref{eq:1}, and thus Theorem~\ref{thm:1}.
\end{proof}

Now we will generalize this argument to prove \eqref{eq:4}.

Let $P(n,k)$ denote the set of triples $(\pi,C,f)$ where $\pi$ is a permutation of $[n]$, $C$ is a cycle of $\pi$ and $f$ is a function from the set of cycles of $\pi$ except $C$ to $[k]$.
The left-hand side of \eqref{eq:4} is equal to
\begin{align*}
\sum_{(\pi,C)\in T(n)} (-k)^{cyc(\pi)-1} &= \sum_{(\pi,C)\in T(n)} (-1)^{cyc(\pi)-1} k^{cyc(\pi)-1}\\
&=  (-1)^{n-1}\sum_{(\pi,C,f)\in P(n,k)} \sgn(\pi),  
\end{align*}
because $\sgn(\pi)=(-1)^{n-cyc(\pi)}$ and for given $(\pi,C)\in T(n)$,
there are $k^{cyc(\pi)-1}$ choices of $f$ with $(\pi,C,f)\in P(n,k)$. 
Thus we get that \eqref{eq:4} is equivalent to the following:
\begin{equation}
  \label{eq:5}
\sum_{(\pi,C,f)\in P(n,k)} \sgn(\pi) = (-1)^{n-k-1} k! (n-k-1)!.
\end{equation}

Let us define the \emph{sign} of $(\pi,C,f)\in P(n,k)$ to be $\sgn(\pi)$.  Let
$\Fix(n,k)$ denote the set of elements $(\pi,C,f)\in P(n,k)$ such that (1)
each integer $i\in[k]$ forms a 1-cycle of $\pi$ and the integers
$k+1,k+2,\ldots,n$ form an $(n-k)$-cycle of $\pi$, which is $C$ and (2) the
$f$ values of the cycles of $\pi$ except $C$ are all distinct.  Then, to
prove \eqref{eq:5}, it is sufficient to find a sign-reversing involution on
$P(n,k)$ whose fixed point set is $\Fix(n,k)$.

We will define a map $\psi:P(n,k)\to P(n,k)$ as follows.
Let $(\pi,C,f)\in P(n,k)$. 

\textbf{\textsc{Case 1:}} There is a pair $(i,j)$ of integers $i<j$ such that $i\in C_1\ne C$ and $j\in C_2\ne C$ with $f(C_1)=f(C_2)$. Here we may have $C_1=C_2$. Let $(i,j)$ be the smallest such pair in lexicographic order. Then we define $\psi(\pi,C,f) = (\tau_{ij}\pi, C, f')$, where $f'(C')=f(C')$ if $i,j\not\in C'$, and $f'(C')=f(C_1)$ otherwise. As before, $\tau_{ij}$ is the transposition exchanging $i$ and $j$.

\textbf{\textsc{Case 2:}} Case 1 does not hold. Then the cycles of $\pi$ except $C$ are all $1$-cycles whose $f$ values are all distinct. Thus there are at most $k$ $1$-cycles of $\pi$ except $C$.

We can represent $(\pi,C,f)$ as a digraph $D$ with vertex set $[n]$ as
follows.  For each integer $i$ contained in $C$, add an edge $i\to\pi(i)$.
For each integer $i$ of $[n]$ which is not contained in $C$, add an edge
$i\to f(i)$, where $f(i)$ is the $f$ value of the $1$-cycle $(i)$
consisting of $i$. For example, see Figures~\ref{fig:digraph} and
\ref{fig:digraph2}.  Note that we can recover $(\pi,C,f)$ from $D$ even
when $D$ consists of cycles only because in this case $C$ is the only cycle
containing integers greater than $k$.

\begin{figure}
  \centering
\begin{tikzpicture}[line width=1pt] 
\node [shift={(1.2,0)}] (1) at (72:1)   {$1$};
\node [shift={(-1,0)}] (8) at (144:1)   {$9$};
\node [shift={(2.4,0)}] (6) at (72:1)   {$11$};
\node (a) at (0:1)     {$3$}; 
\node (b) at (72:1)   {$2$}; 
\node (c) at (144:1) {$8$}; 
\node (d) at (216:1) {$10$}; 
\node (e) at (-72:1) {$5$}; 
\draw [bend left, ->] (b) to (a);
\draw [bend left, ->] (c) to (b);
\draw [bend left, ->] (d) to (c);
\draw [bend left, ->] (e) to (d);
\draw [bend left, ->] (a) to (e); 
\draw [->] (1) to (b); 
\draw [->] (6) to (1); 
\draw [->] (8) to (c); 
\node (3) at (3.5,0) {$4$};
\node (4) at (4.5,0) {$6$};
\node (5) at (5.5,0) {$7$};

\draw [bend left=60,->] (3) to (4);
\draw [bend left=60,->] (4) to (3);
\draw [loop above,->] (5) to (5);
\end{tikzpicture}
 \caption{The digraph representing $(\pi,C,f)\in P(11,8)$, where $\pi=(2,3,5,10,8)(1)(4)(6)(7)(9)(11)$, $C=(2,3,5,10,8)$, $f(1)=2$, $f(4)=6$, $f(6)=4$, $f(7)=7$, $f(9)=8$ and $f(11)=1$.}
  \label{fig:digraph}
\end{figure}
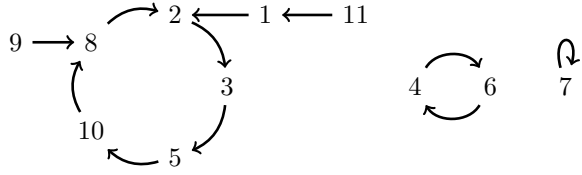

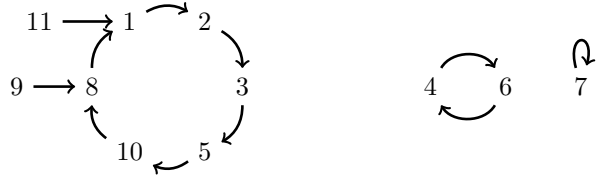
\begin{figure}
  \centering
\begin{tikzpicture}[line width=1pt]
\node (8) at (180:2)   {$9$};
\node [shift = {(-1.2,0)}] (6) at (120:1)   {$11$};
\node (3) at (0:1)     {$3$}; 
\node (2) at (60:1)   {$2$}; 
\node (1) at (120:1)   {$1$};
\node (5) at (180:1) {$8$}; 
\node (9) at (240:1) {$10$}; 
\node (4) at (300:1) {$5$}; 
\draw [bend left, ->] (1) to (2);
\draw [bend left, ->] (2) to (3);
\draw [bend left, ->] (3) to (4);
\draw [bend left, ->] (4) to (9);
\draw [bend left, ->] (9) to (5);
\draw [bend left, ->] (5) to (1); 
\draw [->] (6) to (1); 
\draw [->] (8) to (5); 

\node (x) at (3.5,0) {$4$};
\node (y) at (4.5,0) {$6$};
\node (z) at (5.5,0) {$7$};
\draw [bend left=60,->] (x) to (y);
\draw [bend left=60,->] (y) to (x);
\draw [loop above,->] (z) to (z);
\end{tikzpicture}
 \caption{The digraph representing $(\pi,C,f)\in P(11,8)$, where $\pi=(1,2,3,5,10,8) (4)(6)(7)(9)(11)$, $C=(1,2,3,5,10,8)$, $f(4)=6$, $f(6)=4$, $f(7)=7$, $f(9)=8$ and $f(11)=1$.}
  \label{fig:digraph2}
\end{figure}

Now we consider the two sub-cases where $C$ contains an integer in $[k]$ or not.

\textbf{\textsc{Sub-Case 2-a:}} $C$ does not contain any integer in $[k]$. It is easy to see that we have this sub-case if and only if $(\pi,C,f)\in \Fix(n,k)$. We define $\psi(\pi,C,f) = (\pi,C,f)$.

\textbf{\textsc{Sub-Case 2-b:}} $C$ contains an integer in $[k]$. Let $m$ be the smallest such integer. 

For an integer $i\in C$, we say that $i$ is \emph{free} if $i\in[k]$ and the in-degree of $i$ in $D$ is $1$, i.e.~there is no integer outside of $C$ pointing to $i$.
A sequence $(m_1,m_2,\ldots,m_\ell)$ of integers in $C$ is called a \emph{free chain} if it satisfies (1) for each $i\in[\ell]\setminus\{1\}$, $m_i$ is free and $m_i=\pi(m_{i-1})$, and (2) for each $i\in[\ell]$, $m_i$ is the $i$th-smallest integer in $C$.
Note that we always have a free chain, for example the sequence consisting of $m$ alone. Moreover, there is a unique maximal free chain. 

Let $(m_1,m_2,\ldots,m_\ell)$ be the maximal free chain. Let $\overline m = m_1$ if $\ell$ is odd, and $\overline m = m_2$ if $\ell$ is even. 

\begin{example}
The maximal free chains of the digraphs in Figures~\ref{fig:digraph} and \ref{fig:digraph2} are $(2,3,5)$ and $(1,2,3,5)$ respectively. Thus $\overline m =2$ in both Figures~\ref{fig:digraph} and \ref{fig:digraph2}.
\end{example}

Let $D'$ be the digraph obtained from $D$ by doing the following.  If
$\overline m$ is free, then let $u,v$ be the integers in $C$ such that $D$
has the edges $v\to u$ and $u\to \overline m$. It is not difficult to see
that in this case $C$ has at least two integers, which implies $u\ne
\overline m$. Then we remove the edge $v\to u$ and add an edge
$v\to\overline m$. If $\overline m$ is not free, then let $u$ and $v$ be
the integers with $u\not\in C$ and $v\in C$ such that $D$ has the edges
$u\to \overline m$ and $v\to \overline m$. Then we remove the edge $v\to
\overline m$ and add an edge $v\to u$.

We define $\psi(\pi,C,f)$ to be the element in $P(n,k)$ represented by $D'$.

\begin{example}
Let $(\pi,C,f)$ be represented by the digraph in Figure~\ref{fig:digraph}. Since $\overline m = 2$, $\psi(\pi,C,f)$ is represented by the digraph in Figure~\ref{fig:digraph2}. Note that $\psi(\psi(\pi,C,f))=(\pi,C,f)$.
\end{example}

It is easy to see that $\psi$ is a sign-reversing involution on $P(n,k)$ with fixed point set $\Fix(n,k)$. Thus we have proved \eqref{eq:4} bijectively.

\section*{Acknowledgement}
The author would like to thank the anonymous referee for reading the
manuscript carefully and making helpful comments.  He would also like
to thank Vincent Beck for pointing out a mathematical typo.

\bibliographystyle{elsarticle-num}

\end{document}